\documentclass[a4paper,11pt]{amsart}

\usepackage[T1]{fontenc}

\usepackage{lmodern}

\usepackage[USenglish]{babel}
\usepackage{geometry}

\usepackage[abbrev]{amsrefs}

\usepackage{amssymb}
\usepackage{amsmath}
\usepackage{amsthm}
\usepackage{paralist}

\allowdisplaybreaks

\usepackage{esint}

\newtheorem{theorem}{Theorem}[section]
\newtheorem{lemma}[theorem]{Lemma}
\newtheorem{proposition}[theorem]{Proposition}
\newtheorem{corollary}[theorem]{Corollary}

\newtheorem{definition}[theorem]{Definition}

\newcommand{\N}{\mathbb N}

\newcommand{\R}{\mathbb R}

\newcommand{\mbf}{\mathbf}
\newcommand{\mcal}{\mathcal}

\newcommand{\mrm}{\mathrm}

\renewcommand{\a}{\alpha}

\newcommand{\g}{\gamma}
\newcommand{\G}{\Gamma}
\renewcommand{\d}{\delta}

\newcommand{\e}{\varepsilon}

\renewcommand{\t}{\theta}

\newcommand{\s}{\sigma}

\newcommand{\n}{\nabla}
\newcommand{\fa}{\forall}

\newcommand{\es}{\emptyset}
\newcommand{\wk}{\rightharpoonup}

\newcommand{\us}{\underset}

\newcommand{\bs}{\backslash}
\newcommand{\sm}{\setminus}

\newcommand{\sub}{\subset}

\newcommand{\x}{\times}

\newcommand{\cd}{\cdot}

\newcommand{\ts}{\textstyle}

\renewcommand{\l}{\left}
\renewcommand{\r}{\right}

\newcommand{\bthm}{\begin{theorem}}
\newcommand{\ethm}{\end{theorem}}
\newcommand{\blem}{\begin{lemma}}
\newcommand{\elem}{\end{lemma}}
\newcommand{\bprop}{\begin{proposition}}
\newcommand{\eprop}{\end{proposition}}
\newcommand{\bcor}{\begin{corollary}}
\newcommand{\eor}{\end{corollary}}
\newcommand{\bdefi}{\begin{definition}}
\newcommand{\edefi}{\end{definition}}
\newcommand{\bpf}{\begin{proof}}
\newcommand{\epf}{\end{proof}}
\newcommand{\bl}{\begin{array}{l}}
\newcommand{\bll}{\begin{array}{ll}}
\newcommand{\barr}{\begin{array}}
\newcommand{\earr}{\end{array}}
\newcommand{\bite}{\begin{itemize}}
\newcommand{\eite}{\end{itemize}}
\newcommand{\beq}{\begin{equation}}
\newcommand{\eeq}{\end{equation}}
\newcommand{\beqa}{\begin{eqnarray}}
\newcommand{\eeqa}{\end{eqnarray}}
\newcommand{\beqy}{\begin{eqnarray*}}
\newcommand{\eeqy}{\end{eqnarray*}}

\newcommand{\abs}[1]{\lvert #1 \rvert}
\newcommand{\st}{\;\vert\;}
\newcommand{\dif}{\;\mathrm{d}}

\begin{document}


\title[Groundstates for a class of nonlinear Choquard equations in the plane]%
{Existence of groundstates for a class of nonlinear Choquard equations in the plane}

\author{Luca Battaglia}
\address{Sapienza Universit\`a di Roma\\
Dipartimento di Matematica\\
Piazzale Aldo Moro 5\\
00185 Rome\\
Italy}
\email{battaglia@mat.uniroma1.it}

\author{Jean Van Schaftingen}
\address{Universit\'e catholique de Louvain\\ 
Institut de Recherche en Math\'ematique et Physique\\
Chemin du Cyclotron 2 bte L70.01.01\\
1348 Louvain-la-Neuve\\
Belgium}

\email{Jean.VanSchaftingen@uclouvain.be}

\thanks{This work was supported by the Projet de Recherche (Fonds de la Recherche Scientifique--FNRS) T.1110.14 ``Existence and asymptotic behavior of solutions to systems of semilinear elliptic partial differential equations''.}

\date{December 7, 2016}


\begin{abstract}
We prove the existence of a nontrivial groundstate solution for the class of nonlinear Choquard equation
$$
-\Delta u+u=\big(I_\alpha*F(u)\big)F'(u)\qquad \text{in }\mathbb{R}^2,
$$
where $I_\alpha$ is the Riesz potential of order $\alpha$ on the plane $\mathbb{R}^2$ under general nontriviality, growth and subcriticality  on the nonlinearity $F \in C^1 (\mathbb{R},\mathbb{R})$.
\end{abstract}

\maketitle

\section{Introduction}

We are interested in the existence of nontrivial solutions to the class of nonlinear Choquard equations of the form 
\begin{equation}\label{p}
-\Delta u+u=\bigl(I_\alpha*F(u)\bigr)F'(u)\quad\quad\quad\text{in }\R^N\tag{$\mathcal{P}$},
\end{equation}
where \(N \in \N = \{1, 2, \dotsc\}\), \(\Delta\) is the standard Laplacian operator on the Euclidean space \(\R^N\), $I_\alpha:\R^N \to\R$ is the Riesz potential of order  $\alpha \in(0,N)$ defined for each \(x \in \R^N \setminus \{0\}\) by
\[
I_\alpha(x)=\frac{\Gamma \bigl(\frac{N-\alpha}2\bigr)}{\Gamma \bigl(\frac{\alpha}2\bigr)\pi^\frac{N}22^\alpha |x|^{N-\alpha}},
\]
and a nonlinearity is described by the function $F\in C^1 (\R,\R)$.
Solutions of the equation \eqref{p} are at least formally critical points of the energy functional defined for a function \(u : \R^N \to \R\) by
\begin{equation}\label{iu}
\mathcal{I}(u)=\frac{1}{2} \int_{\R^N}\bigl(|\nabla u|^2+|u|^2\bigr)-\frac{1}{2} \int_{\R^N}\bigl(I_\alpha*F(u)\bigr)F(u).
\end{equation}

In the particular case where for each \(s \in \R\), \(F (s) = s^2/2\), solutions to the Choquard equation \eqref{p} are standing waves solutions of the Hartree equation.
In particular when \(N = 3\) and \(\alpha = 2\), the problem \eqref{p} has arisen in various fields of physics: quantum mechanics \cite{pek}, one-component plasma \cite{lie} and self-gravitating matter \cite{mpt}. 
In these cases, many existence results have been obtained in literature, with both variational \cites{lie,lio,men} and ordinary differential equations techniques \cites{tm,mpt,csv} (see also the review \cite{mv17}).
Such methods extend also to the case of homogeneous nonlinearities \cite{mv13}.

When the nonlinearity \(F\) is not any more homogeneous, it has been shown that the Choquard equation \eqref{p}
has a nontrivial solution if the nonlinearity $F$ satisfies the following hypotheses \cite{mv15}:
\bite
\item[$(F_0')$] there exists $s_0\in\R$ such that $F(s_0)\ne0$;
\item[$(F_1')$] there exists $C>0$ such that $|F'(s)|\le C\bigl(|s|^\frac{\alpha}N+|s|^\frac{\alpha+2}{N-2}\bigr)$ for every $s>0$;
\item[$(F_2')$] $\lim_{s \to 0} {F(s)}/{|s|^{1+\frac{\alpha}N}}=0=\lim_{s \to 0}{F(s)}/{|s|^\frac{N+\alpha}{N-2}}$.
\end{itemize}
The solution \(u\) is a \emph{groundstate}, in the sense that \(u\) minimizes the value of the functional $\mathcal{I}$ among all nontrivial solutions.
The assumptions $(F_0')$, $(F_1')$ and $(F_2')$ are rather mild and reasonable and are ``almost necessary'' in the sense of Berestycki and Lions \cite{bl}: the nontriviality of the nonlinearity condition $(F_0')$ is clearly necessary to have a nontrivial solution; the assumption $(F_1')$ secures a proper variational formulation of the problem \eqref{p} by ensuring that the energy functional \(\mathcal{I}\) is well-defined on the natural Sobolev space \(H^1 (\R^N)\) through the Hardy--Littlewood--Sobolev and Sobolev inequalities; the condition $(F_2')$ is a sort of \emph{subcriticality} condition with respect to the limiting-case embeddings.
The analysis by a Poho\v zaev identity shows that the assumptions $(F_1')$ and $(F_2')$ are necessary in the homogeneous case \(F(s)=s^p/p\) \cite{mv13}.

The results in \cite{mv15} can thus be seen as a counterpart for Choquard-type equations 
of the result of Berestycki and Lions \cite{bl} which give similar ``almost necessary'' conditions for the existence of a groundstate to the equation
\begin{equation}\label{loc}
-\Delta u+u=G'(u)\quad\quad\quad\text{in }\R^N.
\end{equation}
The latter equation can be at least formally be obtained by \eqref{p} by passing to the limit as $\alpha \to0$ and setting $G={F^2}/2$.

Whereas the above-mentioned almost necessary conditions for existence of the Cho\-quard equation \eqref{p} and for the scalar field equation \eqref{loc} have been obtained in higher dimensions \(N \ge 3\), the latter result has been extended to the two-dimensional case \cite{bgk}, under the following assumptions
\bite
\item[$(G_0)$] there exists $s_0\in\R$ such that $G(s_0)>\frac{|s_0|^2}2$;
\item[$(G_1)$] for every $\theta >0$ there exists $C=C_\theta >0$ such that $|G'(s)|\le C_\t \min \{1, s^2\}  e^{\theta|s|^2}$ for every $s>0$;
\item[$(G_2)$] $\lim_{s \to 0} {G(s)/|s|^2}<1/2$.
\end{itemize}
This raises naturally the question whether there is a similar existence result for the Choquard equation \eqref{p} in the planar case.

In the present work, we provide a general existence result for groundstate solutions of problem \eqref{p} in the planar case $N = 2$, which is a two-dimensional counterpart of \cite{mv15} and a counterpart for the Choquard equation of \cite{bgk}. The counterparts of $(F_0'),\,(F_1'),\,(F_2')$ we need are the following:
\begin{itemize}
\item[$(F_0)$] there exists $s_0\in\R$ such that $F(s_0)\ne0$;
\item[$(F_1)$] for every $\theta >0$ there exists $C=C_\theta >0$ such that $|F'(s)|\le C_\theta\min\bigl\{1,|s|^\frac{\alpha}2\bigr\}e^{\theta|s|^2}$ for every $s>0$;
\item[$(F_2)$] $\ts\lim_{s \to 0}{F(s)}/{|s|^{1+\frac{\alpha}2}}=0$.
\end{itemize}
Our main result reads as follows:

\begin{theorem}\label{main}
If $N = 2$ and $F\in C^1(\R,\R)$ satisfies the conditions $(F_0)$, $(F_1)$ and $(F_2)$, then the problem \eqref{p} has a groundstate solution $u\in H^1(\R^2)\setminus \{0\}$, namely the function $u$ solves \eqref{p} and
$$\mathcal{I}(u)=c:=\inf\,\bigl\{\,\mathcal{I}(v)\st v\in H^1\bigl(\R^2\bigr)\bs\{0\}\text{ is a solution of }\eqref{p}\,\bigr\}.$$
\end{theorem}

Let us discuss the assumptions of Theorem~\ref{main}.
As above, the assumption \((F_0)\) is necessary for the existence of a nontrivial solution.
As before, the condition $(F_1)$ ensures needed the well-defineteness of the energy functional on the whole space $H^1(\R^2)$. It has a different shape, because in dimension $N = 2$, the critical nonlinearity for Sobolev embeddings is not anymore a power but rather an exponential-type nonlinearity.
More precisely, the integral of $\min \{1, u^2\} e^{\theta|u|^2}$ on \(\R^2\) is uniformly controlled on $H^1_0(B_1)$ if and only if $\t \int_{B_1}|\nabla u|^2\le 4\pi$ (see \citelist{\cite{mos}\cite{at}}); this is why the parameter $\t > 0$ appears in condition $(F_1)$. 
It will appear that the condition \((F_1)\) is strong enough at infinity. Indeed, by integrating the function \(F'\), it is possible to observe that for every \(\theta > 0\),
\begin{equation}
\label{f}
   \lim_{\vert s \vert \to \infty} \frac{|F(s)|+|F'(s)||s|}{e^{\theta|s|^2}}= 0.
\end{equation}
A subcriticality condition still needs to be imposed around $0$; that is the goal of the subcriticality condition \((F_2)\).

The assumptions $(F_0)$, $(F_1)$ and $(F_2)$ are still \emph{almost necessary}: in the case $F(s)=\frac{s^p}p$, they are satisfied if and only if $\ts p>1+\frac{\alpha}2$, and for $\ts p\le 1+\frac{\alpha}2$ the Choquard equation \eqref{p} has no nontrivial solutions (see \cite{mv13}).\\

In order to prove Theorem~\ref{main} the constraint minimization technique used in \cite{bl,bgk} for the local problem \eqref{loc} does not seem to work, as it introduces a Lagrange multiplier that cannot be absorbed through a suitable dilation because of the presence of three different scalings in the equation and of the nonhomogeneity of the nonlinearity.

Following \cite{mv15}, we use a mountain-pass construction.
We start by constructing a Palais--Smale sequence for the mountain-pass level
\begin{equation}\label{b}
b:=\inf_{\gamma\in\G}\sup_{t\in[0,1]}\mathcal{I}(\gamma(t)),
\end{equation}
where
\begin{equation}\label{gamma}
\Gamma:=\bigl\{\gamma\in C\bigl([0,1],H^1\bigl(\R^2\bigr)\bigr)\st \gamma(0)=0 \text{ and }\mathcal{I}(\gamma(1))<0\bigr\}.
\end{equation}
To avoid relying on an Ambrosetti--Rabinowitz superlinearity condition, we use a scaling trick due to Jeanjean \cite{jea}, which allows to construct Poho\v zaev--Palais--Smale sequence (Proposition~\ref{pps}), namely a Palais--Smale sequence which, in addition, satisfies asymptotically the Poho\v zaev identity
\begin{equation}\label{pu}
\mathcal{P}(u):=\int_{\R^2}|u|^2-\Bigl(1+\frac{\alpha}2\Bigr)\int_{\R^2}\bigl(I_\alpha*F(u)\bigr)F(u)=0.
\end{equation}
Such a condition will imply quite directly the boundedness of the sequence in the space $H^1(\R^2)$ and it will be crucial to get the convergence, hence the existence of a solution (Proposition~\ref{conv}).

We are left with showing that the solution $u$ is actually a groundstate.
To prove this, we first show that the solution $u$ itself satisfies the Poho\v zaev identity (Proposition~\ref{poho}). This will follow by simple calculations once a suitable regularity result is established (Proposition~\ref{reg}); this regularity turns out to be easier to prove from the assumption $(F_1)$ than in the higher-dimensional case \cite{mv15} where a suitable nonlocal Brezis--Kato regularity had to be proved.
The last ingredient that we need is an optimal path \(\gamma_v \in \Gamma \) associated to any solution $v$ of \eqref{p}. The construction of such paths (Proposition~\ref{path}) is inspired by \cite{jt,mv15} but it is more delicate in our two-dimensional case than in the higher dimensions $N\ge3$,  because dilations $t\mapsto v(\cdot/t) \in H^1 (\R^N)$ are not anymore continuous at $t=0$ when $N = 2$.

The content of the paper is the following: in Section~\ref{sectionPreliminaries} we provide some technical preliminaries; in Section~\ref{sectionPPS} we construct the Poho\v zaev--Palais--Smale sequence; in Section~\ref{sectionConvergence} we show that the sequence converges to a solution of \eqref{p}; in Section~\ref{sectionGroundstates} we prove that $u$ is actually a groundstate. In the last section we also state some qualitative result concerning the solutions, which can be proved directly following \cite{mv15}.

\section{Preliminaries}
\label{sectionPreliminaries}

In this section we present some preliminary results which we will need throughout the rest of this paper.
We start by reformulating in a  more convenient form the Moser--Trudinger inequality of Adachi and Tanaka \cite{at}. This quantitative estimate will play a crucial role throughout the paper.

\begin{proposition}[Moser--Trudinger inequality]\label{mt}
For any $\beta\in(0,4\pi)$ there exists $C=C_\beta>0$ such that for every \(u \in  H^1(\R^2)\) satisfying
\[
 \int_{\R^2} |\nabla u|^2\le 1,
\]
one has
$$
  \int_{\R^2}\min\bigl\{1,|u|^2\bigr\}e^{\beta |u|^2}\le C_\beta \int_{\R^2}|u|^2
$$
\end{proposition}

\begin{proof}
The result follows the fact \cite{at}*{Theorem 0.1} that under the conditions of the theorem,
\begin{equation*}
\int_{\R^2}\bigl(e^{\beta |u|^2}-1\bigr)\le C\int_{\R^2}|u|^2.
\end{equation*}
together with the elementary inequalities valid for every \(s \ge 0\),
\[
  \Bigl(1-\frac{1}{e}\Bigr)\max\{1,s\}e^s\le e^s-1\le\max\{1,s\}e^s.\qedhere
\]
\end{proof}

We will also use the Hardy--Littlewood--Sobolev inequality to deal with the nonlocal term (see for example \cite{ll}*{Theorem 4.3}):

\begin{proposition}[Hardy--Littlewood--Sobolev inequality]\label{hls}
For any $p\in [1,\frac{2}\alpha )$ and $f\in L^p(\R^2)$ there exists a constant $C=C_{\alpha,p}$ such that
$$\|I_\alpha \ast f\|_{L^\frac{2p}{2-\alpha p}(\R^2)}\le C\|f\|_{L^p(\R^2)}.$$
\end{proposition}

Combining the last two results with the assumption on $F$ and \eqref{f} we deduce that the energy functional is well-defined on $H^1(\R^2)$:

\begin{proposition}
If $F$ satisfies $(F_1)$, then the energy functional $\mathcal{I}$ defined by \eqref{iu} is well-defined and continuously differentiable.
\end{proposition}

\begin{proof}
We first consider the superposition map \(\mathcal{E}\) defined for each \(u \in H^1 (\R^2)\) and \(x \in \R^2\) by 
\(\mathcal{E}(u) (x) = F' (u(x))\). 
We claim that \(\mathcal{E}\) is well-defined and continuous as a map 
from \(H^1 (\R^2)\) to \(L^{4/\alpha} (\R^2)\). Indeed by assumption \((F_1)\), for every \(\theta > 0\),
and \(s \in \R\), we have  
\[
  \abs{F' (s)}^\frac{4}{\alpha} 
  \le C_\theta^\frac{4}{\alpha} \min \{1, s^2\} e^{\frac{4\theta}{\alpha} \abs{s}^2}.
\]
If \(u \in H^1 (\R^2)\), we take \(\theta > 0\) such that \(\int_{\R^2} \abs{\nabla u}^2 < \frac{\alpha \pi}{2 \theta}\).
We observe that 
\[
\abs{F' (u)}^\frac{4}{\alpha} 
  \le C_\theta^\frac{4}{\alpha} \min \{1, \abs{u}^2\} e^{\frac{4\theta}{\alpha} \abs{u}^2}
\]
on \(\R^2\), 
where the right-hand side is integrable in view of the Moser--Trudinger inequality (Proposition~\ref{mt}); therefore the map \(\mathcal{E}:H^1 (\R^2) \to L^{4/\alpha} (\R^2)\) is well-defined.

If now the sequence \((u_n)_{n \in \N}\) converges to \(u\) in \(H^1 (\R^2)\), then we can assume without loss of generality 
that \(\nu := \sup_{n \in \N} \int_{\R^2} \abs{\nabla u}^2 < \frac{\alpha \pi}{2 \theta}\)
and that \((u_n)_{n \in \N}\) converges to \(u\) almost everywhere.
We have then for some constant \(C \ge 0\),
\[
 C \bigl(\min \{1, \abs{u}^2\} e^{\frac{4\theta}{\alpha} \abs{u}^2} + \min \{1, \abs{u_n}^2\} e^{\frac{4\theta}{\alpha} \abs{u_n}^2}\bigr) 
 - \abs{F' (u) - F' (u_n)}^\frac{4}{\alpha} \ge 0,
\]
for each \(n \in \N\) almost everywhere in \(\R^2\).
By Fatou's lemma we get 
\begin{multline*}
 \liminf_{n \to \infty} 
 \int_{\R^2}   C \bigl(\min \{1, \abs{u}^2\} e^{\frac{4\theta}{\alpha} \abs{u}^2} 
 + \min \{1, \abs{u_n}^2\} e^{\frac{4\theta}{\alpha} \abs{u_n}^2}\bigr)
 - \abs{F' (u) - F' (u_n)}^\frac{4}{\alpha}\\
 \ge 2 C  \int_{\R^2}  \min \{1, \abs{u}^2\} e^{\frac{4\theta}{\alpha} \abs{u}^2}
\end{multline*}
and therefore 
\[
 \limsup_{n \to \infty} \int_{\R^2} \abs{F' (u) - F' (u_n)}^\frac{4}{\alpha}
 \le C \limsup_{n \to \infty} \int_{\R^2} \min \{1, \abs{u_n}^2\} e^{\frac{4\theta}{\alpha} \abs{u_n}^2} - \min \{1, \abs{u}^2\} e^{\frac{4\theta}{\alpha} \abs{u}^2}.
\]
If we consider the set \(A_n^\lambda = \{ x \in \R^2 \st \abs{u_n (x)} \ge \lambda\}\),
we have by Lebesgue's dominated convergence theorem, for every \(\lambda > 0\), 
\begin{multline*}
 \limsup_{n \to \infty} \int_{\R^2 \setminus A_n^\lambda} \min \{1, \abs{u_n}^2\} e^{\frac{4\theta}{\alpha} \abs{u_n}^2}\\
 \le \limsup_{n \to \infty} \int_{\R^2 \setminus A_n^\lambda} \bigl(\min \{1, \abs{u_n}^2\} e^{\frac{4\theta}{\alpha} \abs{u_n}^2} - \min \{1, \abs{u}^2\} e^{\frac{4\theta}{\alpha} \min(\abs{u}^2, \lambda^2)}\bigr) \\
 \shoveright{+ \int_{\R^2} \min \{1, \abs{u}^2\} e^{\frac{4\theta}{\alpha} \abs{u}^2}}\\
 \le \int_{\R^2} \min \{1, \abs{u}^2\} e^{\frac{4\theta}{\alpha} \abs{u}^2}.
\end{multline*}
On the other hand, we have by the Cauchy--Schwarz inequality, the Chebyshev inequality and the Moser--Trudinger inequality (Proposition~\ref{mt})
\[
 \int_{A_n^\lambda} \min \{1, \abs{u_n}^2\} e^{\frac{4\theta}{\alpha} \abs{u_n}^2}
 \le \abs{A_n^\lambda}^\frac{1}{2} \Bigl( \int_{A_n^\lambda} \min \{1, \abs{u_n}^2\} e^{\frac{8\theta}{\alpha} \abs{u_n}^2}\Bigr)^\frac{1}{2}
 \le \frac{C}{\lambda} \int_{\R^2} \abs{u_n}^2.
\]
This allows to conclude that the map \(\mathcal{E} : H^1 (\R^2) \to L^{4/\alpha} (\R^2)\) is continuous.

We now consider the map \(\mathcal{F} : H^1 (\R^2) \to L^{4/(2 + \alpha)} (\R^2)\) defined for each \(u \in H^1 (\R^2)\) by \(\mathcal{F} (u) = F \circ u\).
We observe that for every \(s \in \R\),
\[
 F (s) = \int_0^1 F' (\tau s)s \dif \tau,
\]
and thus for almost every \(x \in \R^2\),
\[
 F (u (x)) = \int_0^1 F'(\tau u (x)) u(x) \dif \tau.
\]
It follows thus from the first part of the proof that \(\mathcal{F}\) is well-defined from \(H^1 (\R^2)\) to \(L^{4/(2 + \alpha)} (\R^2)\).

For the differentiability we consider a sequence \((u_n)_{n \in \N}\) converging strongly to \(u\) in \(H^1 (\R^2)\). We observe that for each \(n \in \N\),
\[
  \mathcal{F} (u_n) - \mathcal{F} (u) - \mathcal{E} (u) (u_n - u)
  = \int_0^1 (\mathcal{E} ((1 - \tau)u + \tau (u_n)) - \mathcal{E} (u)) (u_n - u)\dif \tau,
\]
and thus by H\"older's inequality
\[
   \Vert \mathcal{F} (u_n) - \mathcal{F} (u) - \mathcal{E} (u) (u_n - u)\Vert_{L^{4/(2 + \alpha)}}
   \le \int_0^1 \Vert \mathcal{E} ((1 - \tau)u + \tau (u_n)) - \mathcal{E} (u)\Vert_{L^{4/\alpha}}
   \Vert u_n - u\Vert_{L^2}.
\]
By the convergence of the sequence \((u_n)_{n \in \N}\) and the continuity of the functional \(\mathcal{E}\), it follows that, as \(n \to \infty\),  
\[
   \Vert \mathcal{F} (u_n) - \mathcal{F} (u) - \mathcal{E} (u_n) (u_n - u)\Vert_{L^{4/(2 + \alpha)}}
   = o (\Vert u_n - u\Vert_{L^2}),
\]
that is, \(\mathcal{E}\) represents the Fr\'echet differential of the functional \(\mathcal{F}\).
Since \(\mathcal{E}\) is continuous, it follows that \(\mathcal{F}\) is of class \(C^1\).

Finally, we consider the quadratic form \(\mathcal{Q}\) defined for \(f \in L^{4/(2 + \alpha)}\)
by 
\[
 \mathcal{Q} (f) = \int_{\R^2} (I_\alpha \ast f) f.
\]
By the Hardy--Littlewood--Sobolev inequality (Proposition~\ref{hls}), the quadratic form \(\mathcal{Q}\) is bounded on bounded sets of the space \(L^{4/(2 + \alpha)} (\R^2)\). This implies that \(\mathcal{Q}\) is continuously differentiable and thus the functional 
\[
 u \in H^1 (\R^2) \mapsto \mathcal{Q} (\mathcal{F} (u), \mathcal{F} (u)) = \int_{\R^2} \bigl(I_\alpha \ast F (u)\bigr) F (u)
\]
is continuously differentiable. 
By the smoothness of the norm on a Hilbert space, we conclude that the functional \(\mathcal{I}\) is continuously differentiable.
\end{proof}

Finally, we will use the following improvement of Proposition~\ref{hls} when one has some more $L^p$ integrability:

\begin{proposition}\label{linf}
For any $\ts p\in[1,\frac{2}\a),\,q\in(\frac{2}\alpha,+\infty)$ and $f\in L^p(\R^2)\cap L^q(\R^2)$ there exists $C=C_{\alpha,p,q}$ such that
$$\|I_\alpha \ast f\|_{L^\infty(\R^2)}\le C\bigl(\|f\|_{L^p(\R^2)}+\|f\|_{L^q(\R^2)}\bigr).$$
\end{proposition}

\begin{proof}
The result is classical. We give its short proof for the convenience of the reader.
By choosing $p,q$ in those range we have $\ts(2-\a)\frac{q}{q-1}<2<(2-\a)\frac{p}{p-1}$; therefore, through splitting the integral and H\"older inequality we get for every \(x \in \R^2\)
\begin{equation*}
\begin{split}
|I_\alpha*f(x)|&\le C\int_{\R^2}\frac{|f(x-y)|}{|y|^{2-\alpha}}\mrm dy\\
&\le C\left(\int_{B_1}\frac{\mrm dy}{|y|^{(2-\a)\frac{q}{q-1}}}\right)^{1-\frac{1}q}\|f\|_{L^q\left(B_1(x)\right)}\\
& \hspace{5cm}+C\left(\int_{B_1}\frac{\mrm dy}{|y|^{(2-\a)\frac{p}{p-1}}}\right)^{1-\frac{1}p}\|f\|_{L^p\left(\R^2\sm B_1(x)\right)}\\
&\le C'\bigl(\|f\|_{L^p(\R^2)}+\|f\|_{L^q(\R^2)}\bigr).\qedhere
\end{split}
\end{equation*}
\end{proof}

\section{Construction of a Poho\v zaev--Palais--Smale sequence}
\label{sectionPPS}

In this section we show the existence of a Poho\v zaev--Palais--Smale sequence at the level $b$ defined by \eqref{b}.
In other words, we construct a sequence of almost critical points which asymptotically satisfies the equation \eqref{p} and the Poho\v zaev identity \eqref{pu}.

\begin{proposition}\label{pps}
If the function $F\in C^1(\R,\R)$ satisfies the assumptions $(F_0)$ and $(F_1)$, then there exists a sequence $(u_n)_{n\in\N}$ in $H^1(\R^2)$ such that:
\begin{compactenum}[(a)]
\item $\mathcal{I}(u_n)\us{n \to \infty}\to b$;
\item $\mathcal{I}'(u_n)\us{n \to \infty}\to 0$ strongly in $H^1(\R^2)'$;
\item $\mathcal{P}(u_n)\us{n \to \infty}\to 0$.
\end{compactenum}
\end{proposition}

To prove Proposition~\ref{pps}, we first need to show that the energy functional $\mathcal{I}$ has the mountain pass geometry, namely that the mountain pass level $b$ is well-defined and nontrivial:

\begin{lemma}\label{bb}
The critical level $b$ defined by \eqref{b} satisfies $b\in(0,+\infty)$.
\end{lemma}

\begin{proof}
We start by showing the finiteness of $b$, which will be done as in \cite{mv15}*{Proposition 2.1}.
By the definition of the set $b$, it is sufficient to show that $\Gamma \ne\es$, which in turn is equivalent to find $u_0 \in H^1 (\R^2)$ such that $\mathcal{I}(u_0)<0$.
By the assumption $(F_0)$, we can take $s_0$ such that $F(s_0)\ne0$ and we find
$$\int_{\R^2}\bigl(I_\alpha*F(s_0\mbf1_{B_1})\bigr)F(s_0\mbf1_{B_1})=F(s_0)^2\int_{B_1}\int_{B_1}I_\alpha(x-y)\,\mrm dx\,\mrm dy>0;$$
therefore by density of smooth functions in \(L^q (\R^2)\) there will be $v_0\in H^1(\R^2)$ with $\int_{\R^2}(I_\alpha*F(v_0))F(v_0)>0$. We consider now, for $t>0$, the function \(v_t : \R^2 \to \R\) defined for \(x \in \R^2\) by $\ts v_t(x):=v_0\left(\frac{x}t\right)$. This function verifies
$$
\mathcal{I}(v_t)=\frac{1}{2} \int_{\R^2}|\nabla v_0|^2+\frac{t^2}2\int_{\R^2}|v_0|^2-\frac{t^{2+\alpha}}2\int_{\R^2}(I_\alpha*F(v_0))F(v_0),
$$
therefore, for some $t_0\gg0$, the function $u_0:=v_{t_0}$ satisfies $\mathcal{I}(u_0)<0$.

Let us now show that $b>0$. By the definition of $b$, it is equivalent to show that there exists \(\varepsilon > 0\) such that for every path $\gamma\in\G$ there exists $t_\gamma\in[0,1]$ with $\mathcal{I}(\gamma(t_\g))\ge\e>0$.
We first assume that $u\in H^1(\R^2)$ and $\int_{\R^2}\left(|\nabla u|^2+|u|^2\right)\le\d\ll1$. 
In particular, since $\int_{\R^2}|\nabla u|^2\le 1$, Proposition~\ref{mt} applies to $u$ with $\beta=2\pi$.
Therefore, by Propositions~\ref{hls}, and~\ref{mt} and by \eqref{f}, we have
\begin{equation*}
\begin{split}
\int_{\R^2}\bigl(I_\alpha*F(u)\bigr)F(u)&\le C\left(\int_{\R^2}|F(u)|^\frac{4}{2+\alpha}\right)^{1+\frac{\alpha}2}
\le C\left(\int_{\R^2}\min\bigl\{1,|u|^2\bigr\}e^{2\pi|u|^2}\right)^{1+\frac{\alpha}2}\\
  &\le C\left(\int_{\R^2}|u|^2\right)^{1+\frac{\alpha}2},
\end{split}
\end{equation*}
which is smaller than $\ts\frac{1}4\int_{\R^2}\left(|\nabla u|^2+|u|^2\right)$ if $\d$ is small enough. It follows then that if 
\(\int_{\R^2}\left(|\nabla u|^2+|u|^2\right)\le\d\), we have 
\[
\mathcal{I}(u)\ge\frac{1}4\int_{\R^2}\bigl(|\nabla u|^2+|u|^2\bigr). 
\]
We now take an arbitrary path $\gamma\in\G$. Since $\ts\mathcal{I}(\gamma(1))<0<\frac{1}4\int_{\R^2}\left(|\n\gamma(t_\g)|^2+|\gamma(t_\g)|^2\right)$, we have 
$$
\int_{\R^2}\bigl(|\n\gamma(1)|^2+|\gamma(1)|^2\bigr)>\d>0=\int_{\R^2}\bigl(|\n\gamma(0)|^2+|\gamma(0)|^2\bigr);
$$
therefore, there exists $t_\g \in (0, 1)$ such that $\ts\int_{\R^2}\left(|\n\gamma(t_\g)|^2+|\gamma(t_\g)|^2\right)=\d$, and hence $\ts\mathcal{I}(\gamma(t_\g))\ge\frac{\d}4$. The lemma follows by taking $\ts\e:=\frac{\d}4$.
\end{proof}

\begin{proof}[Proof of Proposition~\ref{pps}]
We follow \citelist{\cite{jea}*{Chapter 2}\cite{hit}*{Chapter 4}\cite{mv15}*{Proposition 2.1}}.
We consider the map \(\Phi\) given by
\beqy
\Phi:\R\x H^1\bigl(\R^2\bigr)&\longrightarrow& H^1\bigl(\R^2\bigr)\\
(\s,v)&\longmapsto&\Phi(\s,v)(x):=v\left(e^{-\s}x\right)
\eeqy
and the functional $\Tilde{\mathcal{I}}=\mathcal{I} \circ \Phi$:
$$\Tilde{\mathcal{I}}(\s,v)=\mathcal{I}(\Phi(\s,u))=\frac{1}{2} \int_{\R^2}|\nabla v|^2+\frac{e^{2\s}}2\int_{\R^2}|v|^2-\frac{e^{(2+\a)\s}}2\int_{\R^2}(I_\alpha*F(v))F(v),$$
which is well-defined and Fr\'echet-differentiable on the Hilbert space $\R\x H^1(\R^2)$.

We define now the class of paths
$$\Tilde{\Gamma}:=\l\{\Tilde{\gamma}\in C\bigl([0,1],\R\x H^1\big(\R^2\bigr)\bigr)\st \Tilde{\gamma}(0)=(0,0) \text{ and } \Tilde{\mathcal{I}}(\Tilde{\gamma}(1))<0\r\};$$
since we have $\G=\bigl\{\Phi \circ \Tilde{\gamma}\st\Tilde{\gamma}\in\Tilde{\Gamma} \bigr\}$, the mountain pass levels of $\mathcal{I}$ and $\Tilde{\mathcal{I}}$ coincide, namely
$$b=\inf_{\Tilde{\gamma}\in\Tilde{\Gamma}}\sup_{t\in[0,1]}\Tilde{\mathcal{I}}(\Tilde{\gamma}(t)).$$
Since, by Lemma~\ref{bb}, the mountain pass level \(b\) is not trivial, we can thus apply the minimax principle (\cite{wil}, Theorem $2.9$) and we find a sequence $((\sigma_n,v_n))_{n\in\N}$ in \(\R\x H^1(\R^2)\) such that:
\begin{align*}
\Tilde{\mathcal{I}}(\sigma_n,v_n)&\us{n \to \infty}\to b&
&\text{ and }&
\Tilde{\mathcal{I}}(\sigma_n,v_n)&\us{n \to \infty}\to0 \text{ strongly in } \bigl(\R\x H^1(\R^2)\bigr)'.
\end{align*}
By writing explicitly the derivative of $\Tilde{\mathcal{I}}$:
$$\Tilde{\mathcal{I}}'(\sigma_n,v_n)[h,w]=\mathcal{I}'(\Phi(\sigma_n,v_n))[\Phi(\sigma_n,w)]+\mathcal{P}(\Phi(\sigma_n,v_n))h;$$
we see that the conclusion follows by taking $u_n=\Phi(\sigma_n,v_n)$.
\end{proof}

\section{Convergence of the Poho\v zaev--Palais--Smale sequence}
\label{sectionConvergence}

In this Section we will construct a nontrivial solution of \eqref{p} from the sequence given by Proposition~\ref{pps}.

\begin{proposition}\label{conv}
If the function $F\in C^1(\R,\R)$ satisfies $(F_1)$ and $(F_2)$ and the sequence $(u_n)_{n\in\N}$ in $H^1(\R^2)$ satisfies
\begin{compactenum}[(a)]
\item $\mathcal{I}(u_n)$ is uniformly bounded,
\item $\mathcal{I}'(u_n)\us{n \to \infty}\to 0$ strongly in $\left(H^1(\R^2)\right)'$,
\item $\mathcal{P}(u_n)\us{n \to \infty}\to 0$;
\end{compactenum}
then, up to subsequences, one of the following occurs:
\begin{itemize}
\item either $u_n\us{n \to \infty}\to0$ strongly in $H^1(\R^2)$;
\item or there exists $u\in H^1(\R^2)\setminus \{0\}$ solving \eqref{p} and a sequence $(x_n)_{n\in\N}$ in $\R^2$ such that $u_n(\cd-x_n)\us{n \to \infty}\wk u$ weakly in $H^1(\R^2)$.
\end{itemize}
\end{proposition}

We follow the strategy of \cite{mv15}*{Proposition 2.2}. Since the gradient does not appear in the Poho\v zaev identity \eqref{pu}, it will be more delicate to show that the nonlocal term does not vanish.

\begin{proof}[Proof of Proposition~\ref{conv}]
We assume that the first alternative does not hold, namely 
\[
  \liminf_{n \to \infty}\int_{\R^2}\bigl(|\nabla u_n|^2+|u_n|^2\bigr)>0.
\]
By writing for each \(n \in \N\)
$$
\frac{1}{2} \int_{\R^2}|\nabla u_n|^2+\frac{\alpha}{2(2+\a)}\int_{\R^2}|u_n|^2=\mathcal{I}(u_n)-\frac{\mathcal{P}(u_n)}{2+\alpha}
$$
we deduce that the sequence $(u_n)_{n\in\N}$ is bounded in the space $H^1(\R^2)$. Since $\mathcal{I}'(u_n)\to 0$ in $H^1 (\R^2)'$ as $n \to \infty$, we have $\mathcal{I}'(u_n)[u_n]\to 0$ as \(n \to \infty\), therefore
$$
  \int_{\R^2}\bigl(I_\alpha*F(u_n)\bigr)F'(u_n)u_n=\int_{\R^2}\bigl(|\nabla u_n|^2+|u_n|^2\bigr)-\mathcal{I}'(u_n)[u_n]\ge\frac{1}C.
$$
Taking $\ts C_0\ge\sup_{n\in\N}\int_{\R^2}\left(|\nabla u_n|^2+|u_n|^2\right)$, we can apply Proposition~\ref{mt} to $\ts\frac{1}{\sqrt{C_0}}u_n$ with $\beta=2\pi$ and we obtain for each \(n \in \N\)
$$
\int_{\R^2}\min\bigl\{1,u_n^2\bigr\}e^{\frac{2\pi}{C_0}|u_n|^2}\le C_{2\pi}\frac{\int_{\R^2}|u_n|^2}{C_0}\le C_{2\pi};
$$
moreover, we also have, as \(n \to \infty\),
\[
\begin{split}
\int_{\R^2}|u_n|^2&=\Bigl(1+\frac{\alpha}2\Bigr)\int_{\R^2}\bigl(I_\alpha*F(u_n)\bigr)F(u_n)+\mathcal{P}(u_n)\\
&=\Bigl(1+\frac{\alpha}2\Bigr)\int_{\R^2}\bigl(I_\alpha*F(u_n)\bigr)F(u_n)+o(1).
\end{split}
\]
Therefore, from Proposition~\ref{hls} and by \eqref{f} we get
\begin{equation}
\begin{split}
\frac{1}C&\le\int_{\R^2}\bigl(I_\alpha*F(u_n)\bigr)F'(u_n)u_n \le C\left(\int_{\R^2}|F(u_n)|^\frac{4}{2+\alpha}\int_{\R^2}\left(|F'(u_n)||u_n|\right)^\frac{4}{2+\alpha}\right)^\frac{2+\alpha}4\\
&\le C'\left(\int_{\R^2}\min\bigl\{1,|u_n|^2\bigr\}e^{\frac{2\pi}{C_0}|u_n|^2}\right)^{1+\frac{\alpha}2} \le C''\left(\int_{\R^2}|u_n|^2\right)^{1+\frac{\alpha}2}\\
&=C''\left(\Bigl(1+\frac{\alpha}2\Bigr)\int_{\R^2}\bigl(I_\alpha*F(u_n)\bigr)F(u_n)+o(1)\right)^{1+\frac{\alpha}2},
\end{split}
\end{equation}
namely $\ts\int_{\R^2}\bigl(I_\alpha*F(u_n)\bigr)F(u_n)$ is bounded from above from zero when \(n \to \infty\).

We now want to prove that $u_n$ does not vanish. We will use the following inequality \cite{lio}*{Lemma I.1} (see also \citelist{\cite{wil}*{lemma 
1.21}\cite{mv13}*{lemma 2.3}\cite{vs14}*{(2.4)}}):
$$\int_{\R^2}|u_n|^p\le C\int_{\R^2}\bigl(|\nabla u_n|^2+|u_n|^2\bigr)\left(\sup_{x\in\R^2}\int_{B_1(x)}|u_n|^p\right)^{1-\frac{2}p}$$
and we will show that the right-hand side term is bounded from below by a positive constant, for every $p>2$.
By the assumption $(F_2)$ and \eqref{f}, for every $\e>0$ there exists $C_{\e,\t}>0$ such that
$$|F(s)|^\frac{4}{2+\alpha}\le\e\min\bigl\{1,|s|^2\bigr\}e^{\theta|s|^2}+C_{\e,\t}|s|^p;$$
therefore
\begin{equation}
\begin{split}
\left(\sup_{x\in\R^2}\int_{B_1(x)}|u_n|^p\right)^{1-\frac{2}p}&\ge\frac{1}C\frac{\displaystyle\int_{\R^2}|u_n|^p}{\displaystyle\int_{\R^2}\bigl(|\nabla u_n|^2+|u_n|^2\bigr)}\\
&\ge\frac{1}{CC_0C_\e}\left(\int_{\R^2}|F(u_n)|^\frac{4}{2+\alpha}-\e\int_{\R^2}\min\bigl\{1,|u_n|^2\bigr\}e^{\frac{2\pi}{C_0}|u_n|^2}\right)\\
&\ge\frac{1}{C'_\e}\left(\left(\int_{\R^2}\bigl(I_\alpha*F(u_n)\bigr)F(u_n)\right)^\frac{2}{2+\alpha}-\e C\int_{\R^2}|u_n|^2\right)\\
&\ge\frac{1}{C'_\e}\left(\frac{1}{C'}-\e CC_0\right).
\end{split}
\end{equation}
The quantity $\e$ being arbitrary, we get $\ts\int_{B_1(x_n)}|u_n|^p\ge\frac{1}C$ for some $x_n\in\R^2$, for $n$ large enough.

We can now consider the translated sequence $(u_n(\cd-x_n))_{n \in \N}$. Since the problem \eqref{p} is invariant by translation, this sequence will satisfy the hypotheses of the present proposition, hence we will still denote it as $(u_n)_{n \in \N}$ and we will assume that $x_n=0$ for all $n\in\N$.
Since $\liminf_{n \to \infty}\int_{B_1}|u_n|^p>0$, we can assume that this sequence \((u_n)_{n \in \N}\) converges weakly to $u\in H^1 (\R^2)\setminus \{0\}$. We just have to show that $u$ solves \eqref{p}.

The sequence $(u_n)_{n\in\N}$ being bounded in $H^1(\R^2)$, the sequence $(F(u_n))_{n\in\N}$ is bounded in $L^p(\R^2)$ for every $\ts p\ge\frac{4}{2+\alpha}$. 
Moreover, up to subsequences, $u_n\to u$ almost everywhere as \(n \to \infty\), so by the continuity of the function $F$ we also have $F(u_n)\to F(u)$ almost everywhere as \(n \to \infty\); this implies that $F(u_n)\wk F(u)$ weakly in $L^p(\R^2)$ for every such $p$ as \(n \to \infty\). 
Since $\ts\frac{2}\a>\frac{4}{2+\alpha}$, by Propositions~\ref{hls} and~\ref{linf} we get $I_\alpha*F(u_n)\wk I_\alpha*F(u)$ weakly in $L^{{4}/{(2-\a)}}(\R^2)\cap L^\infty (\R^2)$ as \({n \to \infty}\).
By the condition $(F_1)$ and Proposition~\ref{mt}, the sequence $(F'(u_n))_{n \in \N}$ is bounded in $L^p(\R^2)$ for every $\ts p\in [\frac{2}\alpha,+\infty)$, and by continuity $F'(u_n)\to F'(u)$ almost everywhere as \(n \to \infty\); therefore, $F'(u_n)\to F'(u)$ strongly in $L^q_{\mathrm{loc}}(\R^2)$ for every $q\in[1,+\infty)$ as $n \to \infty$, hence
$$\bigl(I_\alpha*F(u_n)\bigr)F'(u_n)\us{n \to \infty}\wk\bigl(I_\alpha*F(u)\bigr)F'(u)\quad\quad\quad\text{in }L^r_{\mathrm{loc}}(\R^2)\quad\fa\,r\in\l[1,+\infty\right).$$
Therefore, for every $\varphi\in C^1_0(\R^2)$,
\begin{equation}
\begin{split}
\int_{\R^2}(\nabla u\cdot \nabla \varphi+u\varphi)&=\lim_{n \to \infty}\int_{\R^2}(\nabla u_n\cdot \nabla \varphi+u_n\varphi)\\
&=\lim_{n \to \infty}\int_{\R^2}\bigl(I_\alpha*F(u_n)\bigr)F'(u_n)\varphi =\int_{\R^2}\bigl(I_\alpha*F(u)\bigr)F'(u)\varphi,
\end{split}
\end{equation}
namely $u$ solves the Choquard equation \eqref{p}.
\end{proof}

\begin{corollary}\label{conv2}
If $F$ satisfies the conditions $(F_0)$, $(F_1)$ and $(F_2)$, then problem \eqref{p} has a nontrivial solution $u \in H^1 (\R^2)$.
\end{corollary}

\begin{proof}
By Proposition~\ref{pps}, $\mathcal{I}$ admits a Poho\v zaev--Palais--Smale sequence $(u_n)_{n\in\N}$ at the level $b$.
We apply Proposition~\ref{conv} to $(u_n)_{n\in\N}$. If the first alternative occurred, then we would have $\mathcal{I}(u_n)\to\mathcal{I}(0)=0$ as \(n \to \infty\), in contradiction with Lemma~\ref{bb}. Therefore, the second alternative must occur, and in particular we get a solution $u\in H^1(\R^2)\setminus \{0\}$ of \eqref{p}.
\end{proof}

\section{From solutions to groundstates}
\label{sectionGroundstates}

We start by providing a local regularity result for solution of \eqref{p}. 
This result can be obtained quite directly because our growth assumption $(F_1)$ gives a good control on $I_\alpha*F(u)$ which, in turn, permits to apply a standard bootstrap method. The equivalent result in higher dimension $N\ge3$ is more delicate to prove (see \cite{mv15}*{Theorem 2}) because of the relative weakness of assumption $(F_1')$.

\begin{proposition}\label{reg}
If $F\in C^1(\R,\R)$ satisfies the condition $(F_1)$ and if the function $u\in H^1(\R^2)$ solves  the problem \eqref{p}, then $u\in W^{2,p}_{\mathrm{loc}}(\R^2)$ for every $p \ge 1$.
\end{proposition}

\begin{proof}
By \eqref{f} and Lemma~\ref{mt} we deduce that if $v\in H^1(\R^2)$ then $F(v)\in L^p(\R^2)$ for every $\ts p\ge\frac{4}{2+\alpha}$. Since $\ts\frac{2}\a>\frac{4}{2+\alpha}$, by Proposition~\ref{linf} inequality we get $I_\alpha*F(v)\in L^\infty(\R^2)$.\\
Therefore, any solution $u$ of \eqref{p} verifies
$$\lvert -\Delta u+u\rvert \le C|F'(u)|,$$
with $F'(u)\in L^p_{\mathrm{loc}}(\R^2)$ for every $p \ge 1$ because of $(F_1)$. By standard (interior) regularity theory on bounded domains (see for example \cite{gt}*{Chapter 9}) we deduce that $u\in W^{2,p}_{\mathrm{loc}}(\R^2)$.
\end{proof}

The extra regularity just proved allows to prove that solutions of \eqref{p} satisfy the Poho\v zaev identity \eqref{pu}. 
The proof of the Poho\v zaev identity is classical and it is based on testing \eqref{p} against a suitable  cut-off of $x\cdot \nabla u(x)$, therefore it will be skipped. Details can be found in \cite{mv15}*{Theorem 3}.

\begin{proposition}[Poho\v zaev identity]\label{poho}
If $F\in C^1(\R,\R)$ satisfies $(F_1)$ and $u\in H^1(\R^2)\cap W^{2,2}_{\mathrm{loc}}(\R^2)$ solves \eqref{p}, then
$$\mathcal{P}(u)=\int_{\R^2}|u|^2-\left(1+\frac{\alpha}2\right)\int_{\R^2}\bigl(I_\alpha*F(u)\bigr)F(u)=0.$$
\end{proposition}

The Poho\v zaev identity allows us to show that the mountain pass solution is actually a groundstate. 
We will argue like \citelist{\cite{jt}*{Lemma 2.1}\cite{mv15}*{Proposition 2.1}}, associating to any solution $v$ a path $\gamma_v\in \G$ passing through $v$. The main difficulty here is that the integral of $|\nabla u|^2$ is invariant by dilation, therefore we are not allow to join $v$ with $0$ by just taking dilations $t\mapsto v\left(\frac\cdot t\right)$. To overcome this difficulty, we will combine properly dilatations and multiplication by constants \cite{jt}.

\begin{proposition}\label{path}
If $F\in C^1(\R,\R)$ satisfies $(F_1)$ and $v\in H^1(\R^2)\setminus \{0\}$ solves \eqref{p}, then there exists a path $\gamma_v\in C\left([0,1],H^1(\R^2)\right)$ such that:
\begin{compactenum}[(a)]
\item $\gamma_v(0)=0$;
\item $\gamma_v(1/2)=v$;
\item $\mathcal{I}(\gamma_v(t))<\mathcal{I}(v)$ for every $t\in[0,1]\sm \{1/2 \}$;
\item $\mathcal{I}(\gamma_v(1))<0$.
\end{compactenum}
\end{proposition}

\begin{proof}
We consider the path $\Tilde{\gamma}:[0,+\infty)\to H^1(\R^2)$ given for each \(\tau \in [0, \infty)\) by
$$
(\Tilde{\gamma}(\tau))(x)
:=\begin{cases}
\frac{\tau}{\tau_0}v\bigl(\frac{x}{\tau_0}\bigr)&\text{if }\tau\le\tau_0,\\
v\bigl(\frac{x}\tau\bigr)&\text{if }\tau\ge\tau_0.   
\end{cases}
$$
with $\tau_0\ll1$ to be chosen later. The function $\Tilde{\gamma}$ is clearly continuous on the interval $[0,+\infty)$ and in particular at its boundary $0$. 
For $\tau\ge\tau_0$, Proposition~\ref{poho} gives
\begin{equation}
\begin{split}
\mathcal{I}(\Tilde{\gamma}(\tau))&=\frac{1}{2} \int_{\R^2}|\nabla v|^2+\frac{\tau^2}2\int_{\R^2}|v|^2-\frac{\tau^{2+\alpha}}2\int_{\R^2}(I_\alpha*F(v))F(v)\\
&=\frac{1}{2} \int_{\R^2}|\nabla v|^2+\left(\frac{\tau^2}2-\frac{\tau^{2+\alpha}}{2+\alpha}\right)\int_{\R^2}|v|^2,
\end{split}
\end{equation}
which attains its strict maximum in $\tau=1$ and is negative for $\tau\ge\tau_1$, for some $\tau_1\gg1$.
For $\tau\le\tau_0$ we use \eqref{f} with $\ts\t=\left(1+\frac{\alpha}2\right)\pi$ and then apply Proposition~\ref{mt} to the function $\Tilde{\gamma}(\tau) / \left(\int_{\R^2}|\n\Tilde{\gamma}(\tau)|^2\right)^{{1}/2}$:
\begin{equation}
\label{eqFgTau}
\int_{\R^2}|F(\Tilde{\gamma}(\tau))|^\frac{4}{2+\alpha}\le C\int_{\R^2}\min\l\{1,|\Tilde{\gamma}(\tau)|^2\r\}e^{2\pi|\Tilde{\gamma}(\tau)|^2}\le C\frac{\int_{\R^2}|\Tilde{\gamma}(\tau)|^2}{\int_{\R^2}|\n\Tilde{\gamma}(\tau)|^2}=C\tau_0^2\int_{\R^2}|v|^2, 
\end{equation}
therefore, because of the Poho\v zaev identity (Proposition~\ref{poho}) and the Hardy--Littlewood--Sobolev inequality (Proposition~\ref{hls}), we have 
\[
\begin{split}
\mathcal{I}(\Tilde{\gamma}(\tau))&=\frac{\tau^2}{2\tau_0^2}\int_{\R^2}|\nabla v|^2+\frac{\tau^2}2\int_{\R^2}|v|^2-\int_{\R^2}(I_\alpha*F(\Tilde{\gamma}(\tau)))F(\Tilde{\gamma}(\tau))\\
&\le\frac{1}{2} \int_{\R^2}|\nabla v|^2+\frac{\tau^2}2\int_{\R^2}|v|^2+C\left(\int_{\R^2}|F(\Tilde{\gamma}(\tau))|^\frac{4}{2+\alpha}\right)^{1+\frac{\alpha}2}.
\end{split}
\]
Therefore, in view of \eqref{eqFgTau} and the Poho\v zaev identity again, we deduce that 
\[\begin{split}
\mathcal{I}(\Tilde{\gamma}(\tau))
&\le\frac{1}{2} \int_{\R^2}|\nabla v|^2+\frac{\tau_0^2}2\int_{\R^2}|v|^2+C\tau_0^{2+\alpha}\left(\int_{\R^2}|v|^2\right)^{1+\frac{\alpha}2}\\
&=\mathcal{I}(v)+\left(\frac{\tau_0^2}2-\frac{\alpha}{2(2+\a)}\right)\int_{\R^2}|v|^2+C\tau_0^{2+\alpha}\left(\int_{\R^2}|v|^2\right)^{1+\frac{\alpha}2},
\end{split} 
\]
which is strictly less than $\mathcal{I}(v)$ if $\tau_0=\tau_0(v)$ is chosen small enough. Therefore, the function $\Tilde{\gamma}$ verifies the following properties:
\begin{compactenum}[(a\cprime)]
\item $\Tilde{\gamma}(0)=0$;
\item $\Tilde{\gamma}(1)=v$;
\item $\mathcal{I}(\Tilde{\gamma}(\tau))<\mathcal{I}(v)$ for every $t\in[0,\tau_1]\setminus \{1\}$;
\item $\mathcal{I}(\Tilde{\gamma}(\tau_1))<0$.
\end{compactenum}
To get the required $\gamma_v$ it suffices to take a suitable change of variable $\gamma_v(t):=\Tilde{\gamma}(T(\tau))$ for some function $T \in C ([0, 1], \R)$ satisfying $T(0)=0$, $T(1)=1/2$ and $T(\tau_1)=1$.
\end{proof}

We are now in position to prove the main theorem of this work.

\begin{proof}[Proof of Theorem~\ref{main}]
Let $(u_n)_{n\in\N}$ be the Poho\v zaev--Palais--Smale sequence given by Proposition~\ref{pps}. Then, by Proposition~\ref{conv}, it converges weakly to a solution $u\in H^1(\R^2)\setminus \{0\}$ of \eqref{p}. By definition of groundstate, $\mathcal{I}(u)\ge c$ and, by Proposition~\ref{poho}, we have $\mathcal{P}(u)=0$ (Proposition~\ref{poho} is applicable in view of Proposition~\ref{reg}).
Arguing as in \cite{mv15}*{Theorem 1}, we get successively
\begin{equation}
\begin{split}
\mathcal{I}(u)&=\frac{1}{2} \int_{\R^2}|\nabla u|^2+\frac{\alpha}{2(2+\a)}\int_{\R^2}|u|^2\\
&\le\liminf_{n \to \infty}\left(\frac{1}{2} \int_{\R^2}|\nabla u_n|^2+\frac{\alpha}{2(2+\a)}\int_{\R^2}|u_n|^2\right)=\liminf_{n \to \infty}\left(\mathcal{I}(u_n)-\frac{\mathcal{P}(u_n)}{2+\alpha}\right) =b.
\end{split}
\end{equation}
If $v\in H^1(\R^2)\setminus \{0\}$ is another solution of the Choquard equation \eqref{p}, we apply Proposition~\ref{path} to $v$:
$$\mathcal{I}(v)=\sup_{t\in[0,1]}\mathcal{I}(\gamma_v(t))\ge\inf_{\gamma\in\G}\sup_{t\in[0,1]}\mathcal{I}(\gamma(t))=b.$$
The solution $v$ being arbitrary, by definition of groundstate one has $b\le c$. Putting everything together, we get
$$c\le\mathcal{I}(u)\le b\le c,$$
hence $\mathcal{I}(u)=b=c$. The proof is complete.
\end{proof}

We point out as a corollary of the proof of Theorem~\ref{main}, that the convergence in Proposition~\ref{conv} turns out to be actually a strong convergence in $H^1(\R^2)$ and that this 
gives as a byproduct a compactness property of the set of groundstates of \eqref{p}.

\begin{corollary}
Let $(u_n)_{n\in\N}$ be a Poho\v zaev--Palais--Smale sequence satisfying the assumptions of Proposition~\ref{conv} and in addition
\[
\lim_{n \to \infty}\mathcal{I}(u_n)= c.
\]
Then, there exists $u\in H^1(\R^2)\setminus \{0\}$ solving \eqref{p} and a sequence $(x_n)_{n\in\N}$ in $\R^2$ such that, up to subsequences, $u_n(\cd-x_n)\to_{n \to \infty}u$ strongly in $H^1(\R^2)$.\\
Moreover, the set of groundstates
$$\mcal S_c:=\bigl\{u\in H^1(\R^2);\,u\text{ solves }\eqref{p}\text{ and }\mathcal{I}(u)=c\bigr\}$$
is compact, up to translations, in $H^1(\R^2)$.
\end{corollary}

\begin{proof}
We apply Proposition~\ref{conv}; the first alternative is excluded by our assumption and the continuity of the functional \(\mathcal{I}\) at \(0\). 
Therefore we get, up to translations, $u_n\wk u$ as \(n \to \infty\) in \(H^1 (\mathbb{R}^2)\) and the function $u \in H^1 (\mathbb{R}^2) \setminus \{0\}$ solves \eqref{p}.
As in the proof of Theorem~\ref{main}, we get
\begin{equation}
\liminf_{n \to \infty}\left(\frac{1}{2} \int_{\R^2}|\nabla u_n|^2+\frac{\alpha}{2(2+\a)}\int_{\R^2}|u_n|^2\right)\\
\le c =\frac{1}{2} \int_{\R^2}|\nabla u|^2+\frac{\alpha}{2(2+\a)}\int_{\R^2}|u|^2,
\end{equation}
from which it follows that  $u_n\to u$ strongly in \(H^1 (\R^2)\) as \(n \to \infty\).

To show the compactness of the set of groundstates $\mcal S_c$, we consider an arbitrary sequence $(u_n)_{n\in\N}$ in $\mcal S_c$. Because of Proposition~\ref{poho}, it verifies $\mathcal{P}(u_n)=0$ for every $n \in \mathbb{N}$, so it satisfies the hypotheses of Proposition~\ref{conv} and of the first part of the present corollary; therefore, up to subsequences and translations it will converge to some $u$ which solves \eqref{p} and, by the continuity of the functional $\mathcal{I}$ in $H^1 (\R^2)$, we get $u\in\mcal S_c$.
\end{proof}

We conclude this paper by the following result on additional qualitative properties of the solution \(u\).

\begin{proposition}
If \(F\) is even and nondecreasing on \((0, \infty)\) and \(u\) is a groundstate solution of 
\eqref{p}, then \(u\) has constant sign and is radially symmetric with respect to some point \(a \in \R^N\).
\end{proposition}
\begin{proof}
The proof is the same as \cite{mv15}*{Propositions 5.2 and 5.3}. We briefly sketch the argument for the convenience of the reader.

To prove the constant-sign property, consider the path $\gamma_u$ defined in Proposition~\ref{path}. Since $F$ is an even function, $\mathcal{I}(|v|)=\mathcal{I}(v)$ for every $v\in H^1(\R^2)$, hence $\ts\mathcal{I}(|\gamma_u(t)|)<\mathcal{I}\left(\l|\gamma_u (1 /2)\r|\right)=b$ for every $\ts t\in[0,1]\sm \{1/2\}$. From this, one easily deduces that the function $|u|$ is a groundstate solution of \eqref{p}; since $F'\ge0$, we can apply the strong maximum principle and get $|u|>0$, namely $u$ has constant sign. Without loss of generality we assume now that \(u \ge 0\).

For the symmetry, we follow the strategy of Bartsch, Weth and Willem \cite{bww} and its adaptation to the Choquard equation \cites{mv15,mv13}. For any closed half space $H\sub\R^2$ we consider the reflection $\sigma_H$ with respect to $H$ and define, for every $u\in H^1(\R^2)$, the polarization (see for example \cite{bs})
\[
  u^H(x):=
  \begin{cases}
    \max\{u(x),u(\sigma_H(x))\}&\text{if }x\in H,\\
    \min\{u(x),u(\sigma_H(x))\}&\text{if }x\not \in H.
  \end{cases} 
\]
We first observe that \cite{bs}*{lemma 5.3}
\[
 \int_{\R^2} \abs{\nabla u^H}^2 + \abs{u^H}^2 
 = \int_{\R^2} \abs{\nabla u}^2 + \abs{u}^2
\]
Moreover, since \(F\) is nondecreasing on \((0, + \infty)\), we have \((F \circ u)^H = F \circ (u^H)\) and thus in view of the rearrangement inequality for the Riesz potential
\[
 \int_{\R^2} \bigl(I_\alpha \ast F(u^H) \bigr) F (u^H)
 =  \int_{\R^2} \bigl(I_\alpha \ast F(u)^H \bigr) F (u)^H
 \le \int_{\R^2} \bigl(I_\alpha \ast F(u) \bigr) F (u)
\]
with equality if and only if either \((F \circ u)^H = F \circ u\) or \((F \circ u)^H = F \circ u \circ \sigma_H\) \cite{mv13}*{lemma 5.3}.
It follows thus that, $\mathcal{I}(u^H)\le\mathcal{I}(u)$, with equality holding if and only if either $F(u^H)= F(u)$ or $F(u^H)= F(u \circ \sigma_H)$ on \(\R^2\).
From this and the definition of the level $b$, it follows that $u^H$ is a ground state solutions of \eqref{p}, hence either $F(u^H)= F(u)$ or $F(u^H)= F(u \circ \sigma_H)$ on \(\R^2\). In the former case we easily get $f(u^H)=f(u)$, hence $u^H=u$; in the latter, we similarly get $u^H=u \circ \sigma_H$. The hyperplane $H$ being arbitrary, in either case we conclude that the function $u$ is radially symmetric with respect to some point \(a \in \R^2\) \citelist{\cite{mv13}*{lemma 5.4}\cite{vsw08}*{proposition 3.15}}.
\end{proof}


\end{document}